\documentclass[12pt]{amsart}
\usepackage{amssymb,graphicx,mathtools,amsthm,microtype,mathrsfs,crossreftools,color,amsmath}

\definecolor{darkblue}{rgb}{0.1,0,0.5}
\usepackage[linktocpage=true,colorlinks,linkcolor=darkblue,anchorcolor=green,citecolor=darkblue]{hyperref}

\usepackage{thmtools}
\usepackage[noabbrev,capitalise]{cleveref}

\theoremstyle{theorem}
\newtheorem{theorem}{Theorem}

\newtheorem{lemma}{Lemma}

\theoremstyle{remark}
\newtheorem{remark}[theorem]{Remark}
\AtEndEnvironment{remark}{\hfill$\blacktriangleleft$}

\numberwithin{equation}{section}

\newcommand{\RR}{\mathbb{R}}

\newcommand{\ZZ}{\mathbb{Z}}

\newcommand{\cM}{\mathcal{M}}

\newcommand{\cO}{\mathcal{O}}

\begin{document}

\title[Cycling projections and periodic billiard orbits]{Persistence of periodic billiard orbits under domain deformation} 
\author{Samuel Everett}
\address{University of Chicago}
\email{same@uchicago.edu}

\maketitle

\begin{abstract}
We prove that if a polygon admits a periodic billiard orbit satisfying a certain combinatorial criterion, then there are paths of polygons in parameter space for which every polygon in the path admits a periodic billiard orbit of the same type.
\end{abstract}

\tableofcontents

\section{Introduction}

Consider the uniform motion of a point-mass (billiard) in a simply connected polygonal plane domain $Q\subset \RR^2$, with specular reflection at the boundary. If a trajectory is incident to a vertex of the polygon, its continuation is undefined.
Although such polygonal billiard systems are simple to describe, basic questions are often very difficult to answer.
For instance, it remains unknown whether every polygon admits a periodic billiard trajectory \cite{gutkin2012billiard,schwartz2021survey}.
If a polygon's interior angles are rational multiples of $\pi$, then a result of Masur \cite{masur1986closed} tells us that the polygon admits a periodic billiard trajectory; in fact periodic trajectories are dense in the phase space \cite{boshernitzan1998periodic}. The existence of periodic billiard trajectories in irrational polygons is  best understood (although not completely) in the case of triangles, where there are many positive results \cite{holt1993periodic,gal2003periodic,hooper2007periodic,troubetzkoy2005periodic,schwartz2006obtuse,schwartz2009obtuse,halbeisen2000periodic,vorobets1992periodic}.

The object of this paper is to identify conditions under which a periodic billiard trajectory in a polygon persists (in terms of its type) under deformation of the polygon. Our main theorem is a special case of a result of Troubetzkoy \cite{troub2004recurrence}. The merit of this paper is to provide a new proof and techniques.

Let $\cM_m$ label the parameter space of simply connected $m$-gons of unit area, inheriting a topology from the inclusion $\cM_m\subset (\RR^2)^m$. The \emph{orbit type} of a billiard trajectory $\Gamma$ is the sequence of edges $\{w_i\}$ of $Q$ the trajectory is incident to.

\begin{theorem}\label{thm:main}
Let $Q \in \cM_m$ be a polygon admitting a periodic billiard trajectory $\Gamma$ of orbit type $w = w_{a_1}\cdots w_{a_n}$, where $a_i \in \{1,\dots,m\}$.
Suppose there is an $i=1,\dots, n$ such that the edges $w_{a_i}$ and $w_{a_{i+1}}$ each occur only once in the word $w$ (where $w_{a_{n+1}} = w_{a_1}$).
Then, for every neighborhood $U \subset \cM_m$ of $Q$, there are distinct polygons $P \in U$ and $P' \in \cM_m$ forming the endpoints of a path $\alpha:[0,1]\rightarrow \cM_m$, such that every polygon $R \in \alpha\bigl((0,1)\bigr)\subset \cM_m$ admits a periodic billiard trajectory of orbit type $w$.
\end{theorem}

The regular polygons provide an easy class of examples satisfying the hypothesis of \cref{thm:main}; take the periodic billiard trajectories constructed by joining the midpoints of adjacent edges.

This paper proceeds as follows. In the next section we will introduce the concept of ``cycling projection maps" which are central to the proof of \cref{thm:main}.
Then in \cref{secProofs} we will prove \cref{thm:main}.

\section{Cycling projection maps}\label{secPrelim}

In this section we will review the concept of cycling projection maps introduced in \cite{everett2025geometric}. We develop only the properties necessary to prove \cref{thm:main}.

\subsection{Cycling projection maps}

Let $L_i, L_j \subset \RR^2$ label distinct lines. For every $x \in L_i$ and $\theta \in (0, \pi/2)$, there are two lines $\mathscr{L}_0$ and $\mathscr{L}_1 \subset \RR^2$ such that
\begin{enumerate}
\item[(i)] $\{x\} = \mathscr{L}_0 \cap \mathscr{L}_1$, and
\item[(ii)] $\mathscr{L}_0$ and $\mathscr{L}_1$ intersect with $L_j$ at points $p_0$ and $p_1$, respectively, with acute intersection angle $\theta$.
\end{enumerate}
Refer to \cref{figProjection} for a visual aid.

\begin{figure}
    \centering
    \includegraphics[scale=0.23]{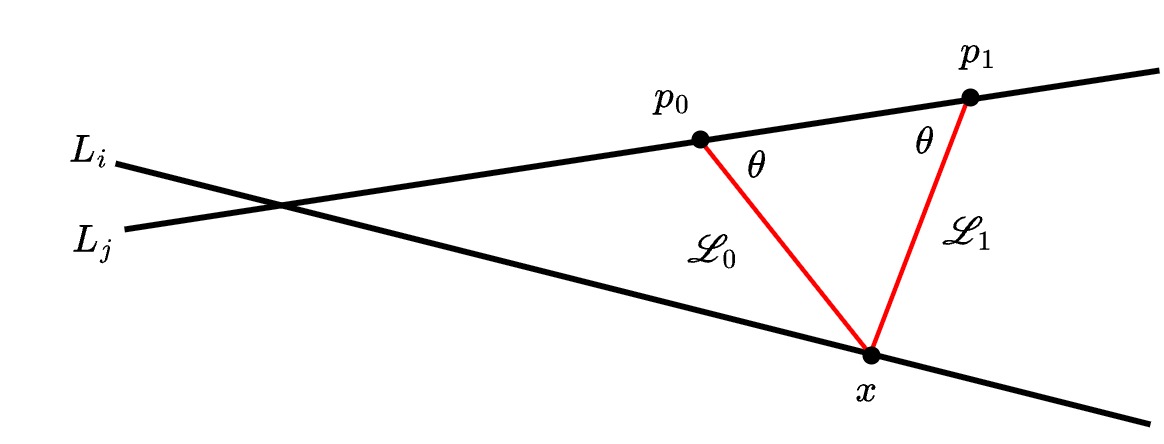}
    \caption{Orientation 0 and 1, angle $\theta$ projections of $x$ onto $L_j$.}
    \label{figProjection}
\end{figure}

We call $p_0, p_1 \in L_j$ the \emph{orientation 0 and 1, angle $\theta$ projections of $x$ onto $L_j$}. We fix the following orientation convention: the orientation $0$ and $1$ projection points are always the ``left" and ``right" points taken from $x$. In other words, the orientation $0$ projection point $p_0$ is always the first projection point encountered when sweeping clockwise about $x$, initially pointing at $p_1$.

Let $X_m \subset \RR^2$ denote the union of $m\geq 3$ nonconcurrent lines in $\RR^2$, labeled $L_1,\dots,L_m$.
Let $o \in \{0,1\}$, $\theta \in (0,\pi/2)$.
An orientation $o$ angle $\theta$ projection
\[
r(o, \theta, L_i):X_m \rightarrow L_i
\]
is a mapping carrying any $x \in X_m$ to its orientation $o$, angle $\theta$ projection on $L_i$. If $x \in L_i$, then $r(o, \theta, L_i)(x) = x$.

For fixed parameters $o, \theta, L_i$, we call $r(o, \theta, L_i)$ a \emph{projection rule}. Interesting behavior arises when many of these projection rules are composed in a fixed, cycling order.

A \emph{projection rule sequence} associated with a space $X_m$ is a sequence of $n \geq m\geq 3$ projection rules, denoted $\{r_i\}_{i=1}^n\coloneqq \{r_i(o_i,\theta_i, L_{a_i})\}_{i=1}^n$, $a_i \in \{1,\dots,m\}$, satisfying the property that consecutive projection rules in the sequence do not project onto the same line, including the first and last.

We will cyclically iterate the projection rules of a sequence $\{r_i\}_{i=1}^n$ in order to obtain a dynamical system.
A \emph{cycling projection map} $T_n:X_m\rightarrow X_m$ is defined to be a cycling composition of $n\geq 3$ projection rules in an associated defining projection rule sequence $\{r_i\}_{i=1}^n$. More precisely, for $x \in X_m$, define iteration of $T_n$ by:
\[
T_n(T^{n+1}_n(x)) = T^{n+2}_n(x) = r_2(r_1(r_n(\dots r_2(r_1(x))))).
\]
See \cref{figExampleIteration} for a demonstration of what iterating a cycling projection map looks like.

\begin{figure}
    \centering
    \includegraphics[scale=0.17]{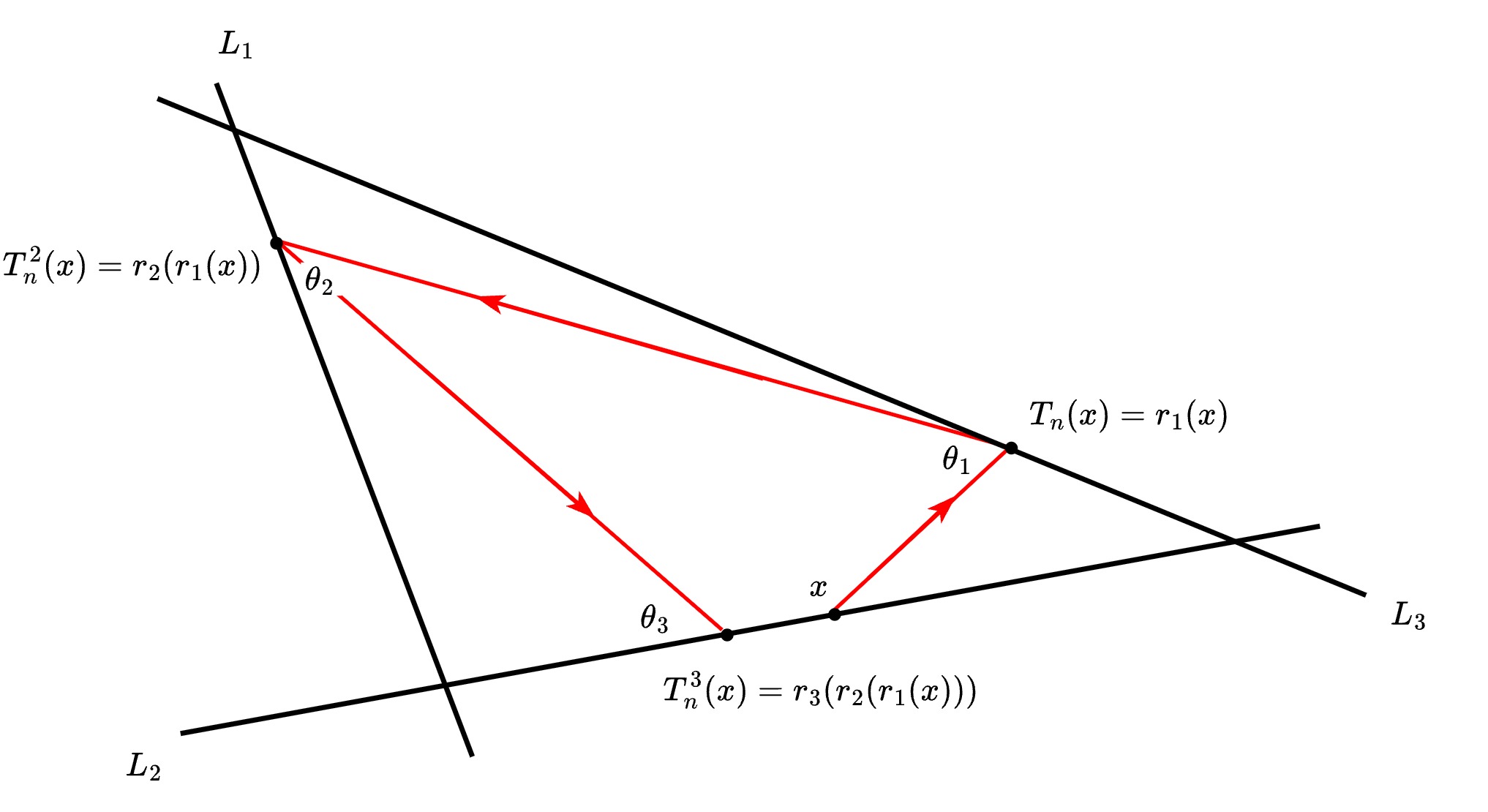}
    \caption{Example of iterating a cycling projection map $T_n$ in a space $X_3$, whose first three projection rules $r_1,r_2,r_3$ in its defining rule sequence have projection angles $\theta_1,\theta_2,\theta_3$, orientation values $1, 1, 0$, and project onto lines $L_3, L_1, L_2$, respectively.}
    \label{figExampleIteration}
\end{figure}

\subsection{Asymptotic behavior of cycling projection maps}
Let $(X_m, T_n)$ label a dynamical system.
The \emph{orbit} $\cO_{T_n}(x)$ of a point $x \in X_m$ with respect to $T_n$ is the set $\{x, T_n(x), T^2_n(x),\dots\}$.

A cycling projection map $T_n:X_m\rightarrow X_m$ with defining rule sequence $\{r_i\}_{i=1}^n$ is called \emph{redundant} if there exists a $k<n$ and cycling projection map $T'_k:X_m\rightarrow X_m$ with defining rule sequence $\{r'_i\}_{i=1}^k$ such that $\cO_{T_n}(x) = \cO_{T'_k}(x)$ for all $x \in X_m$. We shall assume moving forward that the cycling projection maps considered are not redundant.

\begin{remark}\label{rem:fixedLineOrder}
Let $T_n:X_m\rightarrow X_m$ be a cycling projection map. Each projection rule in its defining rule sequence always projects onto the same line. As a consequence, after the first iteration of the map, the itinerary (by line visited) of any point in $X_m$ under iteration of the map will cycle between a subset of the lines composing $X_m$ in the same order.
\end{remark}

Projection rules, when restricted to mapping from one line to another, are similitudes.
More precisely, let $L_1, L_2 \subset \RR^2$ be two lines, let $d$ be the Euclidean metric, and let $r\coloneqq r(o, \theta, L_2)$ be a projection rule. If $L_1,L_2$ are parallel, it is immediate that $d(r(x), r(y)) = d(x, y)$ for all $x, y \in L_1$.

If $L_1$ and $L_2$ intersect, an elementary law of sines argument (see \cite{everett2025geometric}) establishes the following facts:
\begin{enumerate}
\item[(i)] For all distinct $x, y \in L_1$, $d(r(x), r(y)) = cd(x, y)$ for some $c>0$.
\item[(ii)] For a fixed orientation $o \in \{0,1\}$, there are at most two values  $\alpha_1,\alpha_2 \in (0,\pi/2)$, $0<\alpha_1<\alpha_2<\pi/2$, such that for all values $\theta$ in one of the intervals $(0,\alpha_1)$, $(\alpha_1,\alpha_2)$, $(\alpha_2, \pi/2)$, it holds that either $c<1$ or $c>1$, and $c=1$ iff $\theta = \alpha_i$, $i=1,2$.
\end{enumerate}

We call the constant $c$ associated to the projection rule $r$ a \emph{similarity coefficient}.
Notice the value of $c$ depends on both the projection angle $\theta$ and the least angle between the lines $L_1$ and $L_2$.

As noted in \cref{rem:fixedLineOrder}, after the first iteration of a cycling projection map $T_n$, each projection rule in the defining rule sequence of $T_n$ will always map between the same two lines in $X_m$. Hence, iteration of a cycling projection map is just a cycling composition of similitudes.

Define $\hat{T}_n \coloneqq T^n_n$ to be the induced map over a line $L\subset X_m$, so that $\hat{T}_n^k = T^{kn}_n$. If the rules in the defining projection rule sequence for $T_n$ have associated similarity coefficients $c_1,\dots,c_n$, then let $C = c_1c_2\cdots c_n >0$ denote the similarity coefficient for the induced map $\hat{T}_n$.

\begin{lemma}[\cite{everett2025geometric}]\label{lem:contraction}
Let $T_n:X_m\rightarrow X_m$ be a cycling projection map and $\hat{T}_n$ its associated induced map with similarity coefficient $C$.
If $C <1$, then $T_n$ admits a unique, globally attracting periodic orbit of prime period $n$.
\end{lemma}
\begin{proof}
The assertion is an immediate consequence of the Banach fixed point theorem: if $C<1$ then $\hat{T}_n$ admits a unique (globally attracting) fixed point, so $T_n$ admits a unique attracting periodic orbit.
\end{proof}

See \cref{figNumericalSim} for a numerical simulation demonstrating the result.
Notice the orbit of a cycling projection map generates a polygonal path contained in $\RR^2$ when the consecutive points of the orbit are joined by a line segment.

\begin{figure}
    \centering
    \includegraphics[scale=0.09]{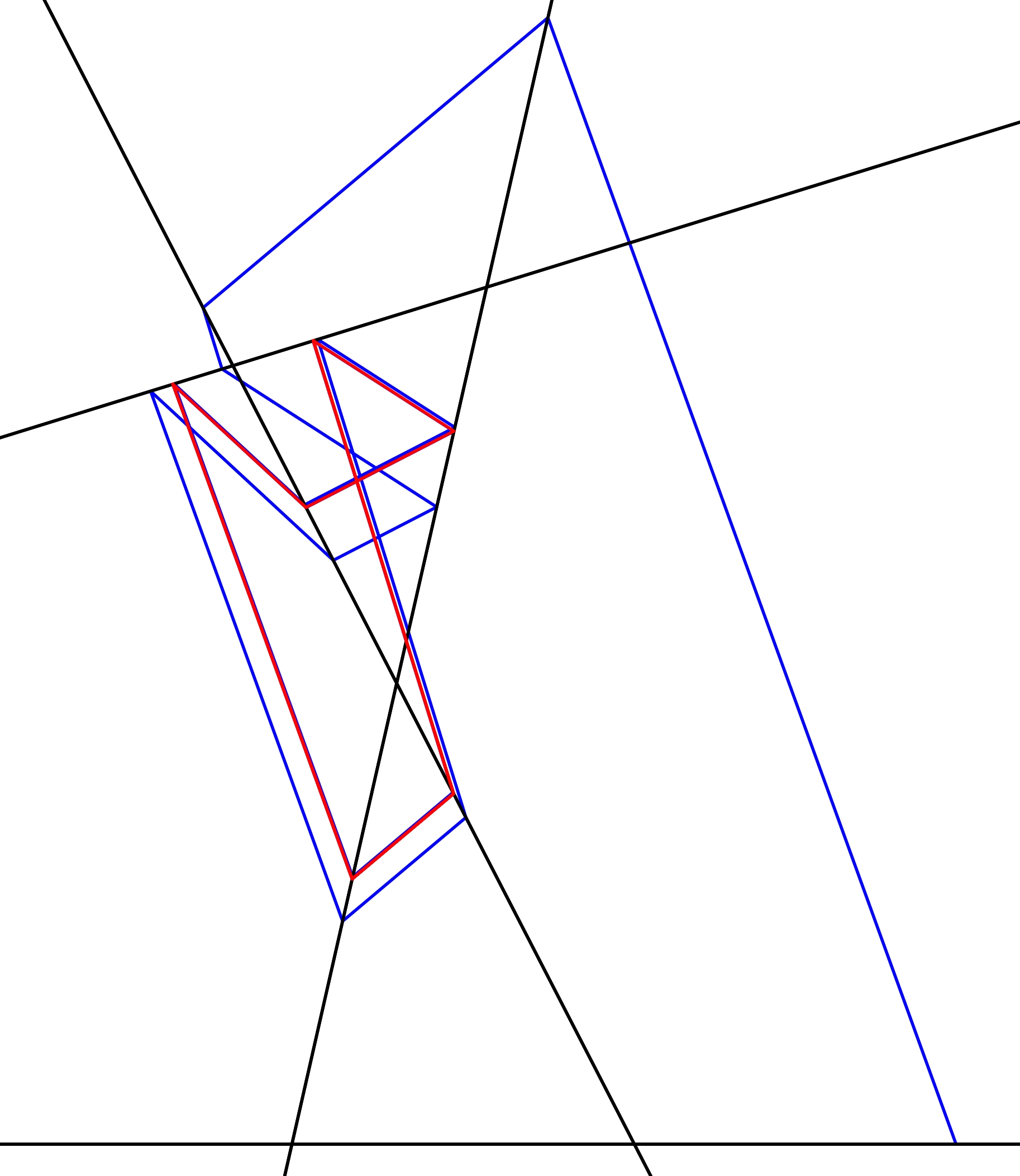}
    \caption{A numerical demonstration of how iteration of a cycling projection map with six defining rules converges to a periodic orbit. The blue line segments link the points of the orbit, and the red line segments link the periodic points.}
    \label{figNumericalSim}
\end{figure}

\begin{lemma}\label{lem:inducedIsContinuous}
Let $T_n:X_m\rightarrow X_m$ be a cycling projection map with defining rule sequence $\{r_i\}_{i=1}^n$ whose angle projection parameters are $\overline{\theta} = (\theta_1,\dots,\theta_n)$. Let
\[
\hat{T}_{n,\overline{\theta}}:L\rightarrow L
\]
denote its associated induced map of a line $L \subset X_m$.
The function
\[
F:(0,\pi/2)^n \times L \rightarrow L
\]
defined so that $F(\overline{\theta}, x) = \hat{T}_{n,\overline{\theta}}(x)$, is jointly continuous on $(0,\pi/2)^n \times L$.
\end{lemma}
\begin{proof}
For a fixed orientation $o$ and line $L \subset X_m$, when treating the projection angle $\theta$ as a variable, a projection rule $r:(0,\pi/2)\times X_m\rightarrow L$ is immediately seen to be jointly continuous over $(0,\pi/2)\times X_m$. So is their composition, and the assertion follows.
\end{proof}

\subsection{Parameter variation and fixed points}

The similarity coefficient $C$ of the induced map $\hat{T}_n$ equals the product $c_1c_2\cdots c_n$ of the associated similarity coefficients of the rules from the defining rule sequence.
The similarity coefficients of the rules, we have established, vary continuously with respect to the projection angle $\theta$ of their associated rule, once an orientation value $o \in \{0,1\}$ for the rule has been fixed.
Consequently, $C = C(\theta_1,\theta_2,\dots,\theta_n)$ is a continuous function of the parameters $(\theta_1,\dots,\theta_n) \in (0,\pi/2)^n\subset \RR^n$.

\begin{remark}\label{rem:lineRotation}
In the same way the behavior of a system $(X_m, T_n)$ is dependent on the projection angles of the rules defining $T_n$, adjusting the lines defining $X_m$ changes the similarity coefficients $c_1,\dots,c_n, C$.
\end{remark}

With this remark in mind, we have the following lemma; it is trivial, but worth stating.

\begin{lemma}\label{lem:lineRotation}
Let $T_n:X_m\rightarrow X_m$ be a cycling projection map, and let $\hat{T}_n:L\rightarrow L$ be the induced map over some line $L \subset X_m$, with similarity coefficient $C$.
Suppose $L_i \subset X_m$ is a line that exactly one rule defining $T_n$ projects onto.
If $L_i$ is rotated by a sufficiently small amount about any of its points, then under the new space $X_m'$ with the perturbed line, the updated similarity coefficient $C'$ of $\hat{T}_n$ no longer equals $C$.
\end{lemma}
\begin{proof}
The rotation of a line $L_i \subset X_m$ must perturb the similarity coefficient $c_i$ of the single rule projecting onto $L_i$. But $C=c_1\cdots c_n$, so if one of the $c_i$ changes, so does $C$.
\end{proof}

\begin{lemma}\label{lem:contFP}
Let $P = (0,A)^n\subset \RR^n$. Suppose $f:P\times \RR\rightarrow \RR$ is a jointly continuous function such that for each $p \in P$, $x\mapsto f(p, x)$ is a contraction on $\RR$. Then the unique fixed point $x^*(p) \in \RR$ is continuous with respect to $p$ on $P$.
\end{lemma}
\begin{proof}
Let $p \in P$ be arbitrary. The function $f_p: \mathbb{R} \to \mathbb{R}$ defined by $f_p(x) = f(p, x)$ is a jointly continuous contraction mapping on the complete metric space $\mathbb{R}$, with a contraction constant $L=L(p) < 1$. By the contraction mapping theorem, there is a unique $x^*(p) \in \mathbb{R}$ such that $f(p, x^*(p)) = x^*(p)$. So, the mapping $x^*: P \to \mathbb{R}$ is well-defined.
Define $G:P\times\mathbb R\to\mathbb R$ by $G(p,x):=f(p,x)-x$. Then $G$ is also jointly continuous.

Take $p\in(0,1)^n$. If $x>y$, and $L(p)\in[0,1)$ is a contraction constant for $f_p$, then
\begin{align*}
G(p,x)-G(p,y)
&=\bigl(f(p,x)-f(p,y)\bigr)-(x-y)\\
&\le |f(p,x)-f(p,y)|-(x-y)\\
&\le (L(p)-1)(x-y)\\
&<0,
\end{align*}
so the function $x\mapsto G(p,x)$ is strictly decreasing and continuous. In particular, since $G(p,x^*(p))=0$ and the zero is unique, it follows that $G(p,x)>0$ for $x<x^*(p)$, and $G(p,x)<0$ for $x>x^*(p)$.

To prove continuity of $x^*$, fix $p_0\in P$ and set $x_0=x^*(p_0)$. Given $\varepsilon>0$, let $a = x_0-\varepsilon$ and $b=x_0+\varepsilon$. By the foregoing sign characterization at $p_0$, we have $G(p_0,a)>0$ and $G(p_0,b)<0$. By joint continuity of $G$, there exists $\delta>0$ such that $\|p-p_0\|<\delta$ implies $G(p,a)>0$ and $G(p,b)<0$.

For such $p$, continuity of $x\mapsto G(p,x)$ and the intermediate value theorem yields the existence of some $x\in(a,b)$ with $G(p,x)=0$. Uniqueness of the zero (equivalently, uniqueness of the fixed point of the contraction $f_p$) forces $x=x^*(p)$, hence $|x^*(p)-x_0|<\varepsilon$. Since $\varepsilon>0$ was arbitrary, $x^*$ is continuous at $p_0$, and because $p_0$ was arbitrary, the map $p\mapsto x^*(p)$ is continuous on $P$.
\end{proof}

\section{Proof of \cref{thm:main}}\label{secProofs}

We recall some terminology.
A \emph{billiard trajectory} $\Gamma$ in a simply connected $m$-gon $Q \in \cM_m$ is a polygonal path (generally infinite) composed of line segments $\{l_i\} \subset Q$ so that each vertex of the path $l_i \cap l_{i+1}$ lies in the interior of some edge of $Q$. A billiard trajectory $\Gamma$ is \emph{periodic} if it closes up.
Let $\partial Q$ denote the boundary (edges) of a polygon $Q \in \cM_m$.

The \emph{billiard map} $S$ is the first return map of the billiard flow to the boundary $\partial Q$. The phase space of the billiard map is the set of inward-pointing unit vectors with \emph{foot point} $q$ in $\partial Q$. The direction $\theta$ of a vector $z = (q, \theta)$ will refer to the angle the vector makes with the clockwise direction of the boundary.

\subsection{Proof of \cref{thm:main}}

Before stating the proof of \cref{thm:main}, we give some intuition.

In the previous section, we demonstrated that whenever a cycling projection map $T_n$ has an induced map with similarity coefficient $C<1$, then trajectories are asymptotically stable, converging to a periodic orbit. In particular, if $C<1$, then under sufficiently small perturbations to the lines defining the underlying space $X_m$, or the angle projection parameters $\theta_i$ of the defining rules $r_i$, it continues to be the case that $C<1$ and orbits are asymptotically periodic. Furthermore, when $C<1$, small variation in either the lines defining $X_m$ or rule projection angles $\theta_i$ also slightly varies the associated periodic orbit (\cref{lem:contFP}).

The strategy of the following proof is to exploit this stability.
The proof strategy is perhaps best communicated in a picture, so we suggest the reader consult \cref{figProofDiagram} before continuing to the proof to gain intuition.

\begin{figure}
    \centering
    \includegraphics[scale=0.315]{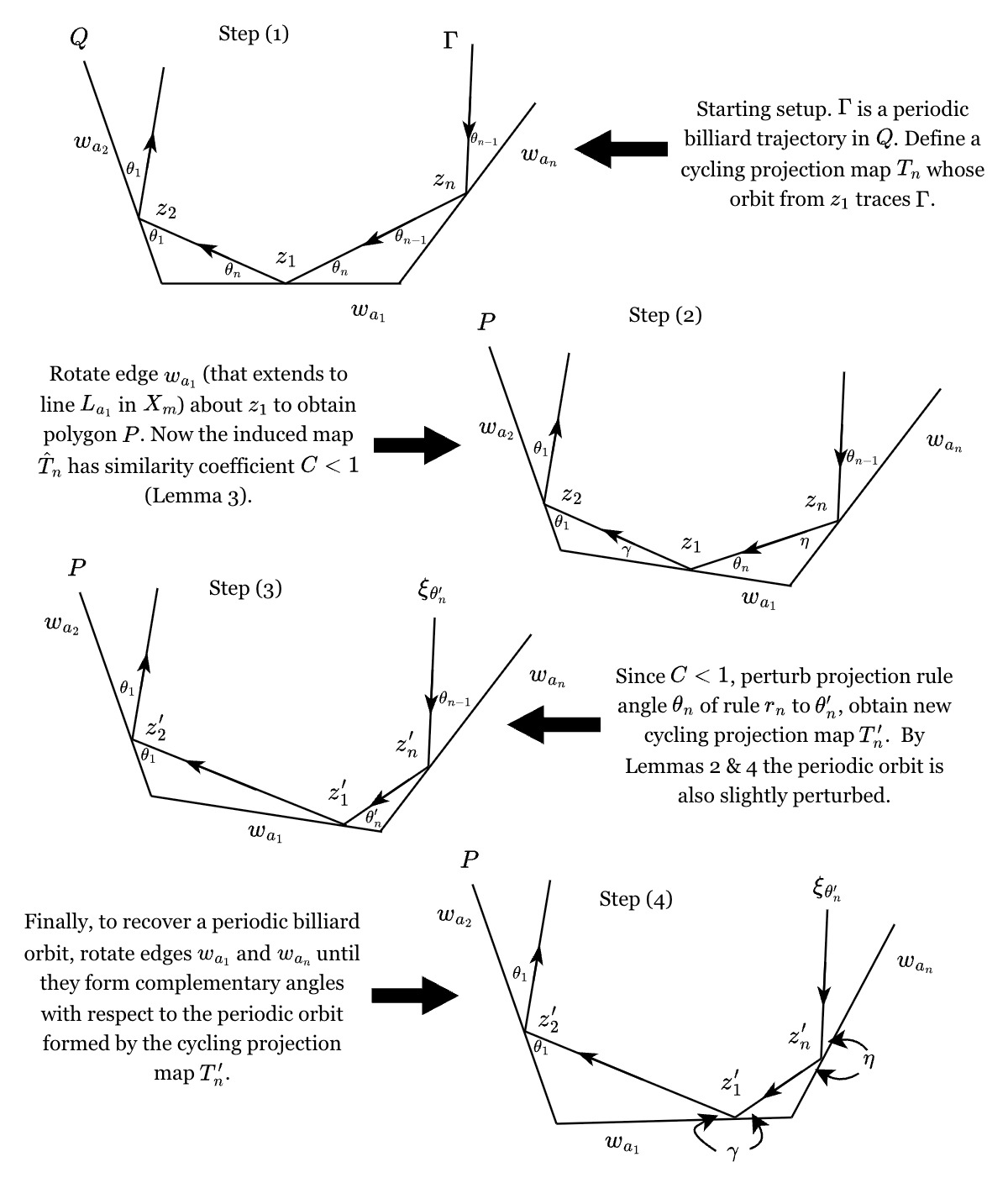}
    \caption{}
    \label{figProofDiagram}
\end{figure}

\begin{proof}[Proof of \cref{thm:main}]
Let $Q \in \cM_m$ be a polygon admitting a periodic billiard trajectory $\Gamma$ of orbit type $w = w_{a_1}\cdots w_{a_n}$.
Suppose there is an $i=1,\dots, n$ such that edges $w_{a_i}$ and $w_{a_{i+1}}$ only appear once in $W$ (using the convention $w_{a_{m+1}} = w_{a_1}$). Notice this assumption rules out ``perpendicular" periodic trajectories.
Let $S$ denote the billiard map. Put $z_{t+1} = S^t(z_1)$, $t=1,\dots,n-1$, and $z_1 = S^{m}(z_1)$; these are the periodic points of the billiard trajectory $\Gamma$.

Without any loss of generality, let $z_1 \in w_{a_1}$ be the only periodic point (vertex) of $\Gamma$ contained in edge $w_{a_1}$, and $z_n \in w_{a_n}$ the only periodic point of $\Gamma$ contained in edge $w_{a_n}$. Then there is no $k$, $1\leq k \leq n-2$, such that $S^k(z_1) \in w_{a_1}$ or $S^k(z_1) \in w_{a_n}$ (by hypothesis).

Extend the line segments composing the edges of $Q$ to obtain $m$ lines $L_{w_i}$ which compose a space $X_m$. Then $\partial Q \subset X_m$. 
Define a cycling projection map $T_n:X_m\rightarrow X_m$ with defining rule sequence $\{r_i\}_{i=1}^n$ such that $r_{i}(z_{i}) = z_{i+1}$, $i=1,\dots, n-1$, and $r_n(z_{n}) = z_1$.
With $z_1 \in w_{a_1}$, let $\hat{T}_n:L_{w_{a_1}}\rightarrow L_{w_{a_1}}$ be the associated induced map along line $L_{w_{a_1}}$.

If $C$ is the similarity constant for $\hat{T}_n$, then, as constructed, $C=1$. This follows from the fact that polygonal billiard systems are conservative (see \cite{everett2025geometric} for discussion).
Using \cref{lem:lineRotation}, rotate line $L_{w_{a_1}}$ about point $z_1$ by a sufficiently small amount so that $C<1$. We may perform such a rotation because (i) $z_1$ is the only periodic point contained on the edge $w_{a_1}$, and (ii) any periodic point (vertex) of $\Gamma$ must lie in the interior of the edges adjacent to $w_{a_1}$, so a sufficiently small rotation will not destroy the vertices of the periodic orbit $\Gamma$.

Let $P$ label the polygon resulting from this perturbation of $Q$. Carry over the notation for edges and periodic points from $Q$ to $P$ in the natural way.

By construction, $z_1 \in L_{w_{a_1}}$ is a fixed point for $\hat{T}_n$. Moreover, by \cref{lem:contraction} this fixed point is unique.
Treating the projection rule angles $\theta_1,\dots,\theta_n$ as variables, \cref{lem:inducedIsContinuous} ensures the function
\[
F:(0,\pi/2)^n \times L_{w_1} \rightarrow L_{w_1}
\]
defined so that $F(\overline{\theta}, x) = \hat{T}_{n,\overline{\theta}}(x)$, is jointly continuous on $(0,\pi/2)^n \times L_{w_1}$.
In addition, \cref{lem:contFP} states the fixed point varies continuously with small variation of the angle projection parameters $\theta_1,\dots,\theta_n$.
Suppose the angle projection parameter of rule $r_n$ in the defining rule sequence of $T_n$, for which $r_n(z_{n}) = z_1$, is $\theta_n>0$.

Then, collecting the above facts, there is an $\epsilon>0$ such that for every $\theta_n' \in (\theta_n-\epsilon, \theta_n + \epsilon)$, $\hat{T}_n'$ has a fixed point in a neighborhood of $z_1$, where $\hat{T}'_n$ is the induced map of $T'_n$ --- identical to $T_n$ but for the last rule $r_n$ whose projection angle is $\theta_n'$.
In particular, by \cref{lem:contFP} this fixed point varies continuously with $\theta_n' \in (\theta_n - \epsilon, \theta_n + \epsilon)$.

Hence, corresponding to each $\theta_n' \in (\theta_n - \epsilon, \theta_n + \epsilon)$ we obtain a closed curve $\xi_{\theta_n'} \subset P$ by joining consecutive updated periodic points $z_{i}'$ (periodic points of the cycling projection map $T_n'$), $i=1,\dots,n$ with line segments.
Taking $\epsilon$ sufficiently small, we are guaranteed that $\xi_{\theta_n'} \subset P$ even if $P$ is not convex, since the original trajectory $\Gamma$ is not incident to any vertices of $Q$.

Only the pairs of consecutive line segments $\overline{z_2'z_1'}$, $\overline{z_1'z_n'}$, and $\overline{z_1'z_n'}$, $\overline{z_n'z_{n-1}'}$ do not form complementary angles with respect to edges $w_{a_1}$ and $w_{a_n}$, respectively. The rest of the closed curve $\xi_{\theta_n'}$ does satisfy the mirror law of reflection with respect to $P$, because the remaining projection angles of the cycling projection map have remained fixed, and no other edge of $Q$ has been changed to obtain $P$ other than $w_{a_1}$.

Hence, for each $\theta_n'$, we obtain a polygon $E$ admitting a periodic billiard trajectory of orbit type $w$ as follows. Rotate the edges $w_{a_1}$ and $w_{a_n}$ so that the line segments $\overline{z_2'z_1'}$, $\overline{z_1'z_n'}$, and $\overline{z_1'z_n'}$, $\overline{z_n'z_{n-1}'}$ now form complementary angles $\gamma$ and $\eta$ with respect to the edges $w_{a_1}$ and $w_{a_n}$, respectively. This operation recovers a periodic billiard trajectory in a polygon $E$ nearby $P$ in parameter space. Recall this rotation operation is possible because there remains only one periodic point on the edges $w_{a_1}$ and $w_{a_n}$.

The family of closed curves $\xi_{\theta_n'}$ is parameterized by a single real value $\theta_n' \in (\theta_n - \epsilon, \theta_n + \epsilon)$, and these closed curves have identical itinerary to $\Gamma$ with respect to the edges of the relevant polygon (taking the same edge labels). Furthermore, we have seen that for each of these curves $\xi_{\theta_n'}$ we may transform $P$ to recover a periodic billiard trajectory. In addition, the perturbation to $Q$ from which we recovered $P$ may be arbitrarily small.

Collecting these observations, we find that for every neighborhood $U \subset \cM_m$ of $Q$, there are distinct polygons $E \in U$ and $E' \in \cM_m$ forming the endpoints of a path $\alpha:[0,1]\rightarrow \cM_m$, such that every polygon $R \in \alpha\bigl((0,1)\bigr)\subset \cM_m$ admits a periodic billiard trajectory of orbit type $w$.
\end{proof}

Recall that a periodic billiard trajectory $\Gamma$ in $Q$ is called \emph{stable} if there is a neighborhood $U \subset \cM_m$ of $Q$ consisting of $m$-gons that have a periodic billiard orbit with the same orbit type as $\Gamma$.
A theorem of Vorobets, Gal'perin, and Stepin \cite{vorobets1992periodic} asserts that if a polygon $Q$ has interior angles $\alpha_1,\dots,\alpha_n$ that are rationally independent, i.e.
\[
(k_1,\dots,k_n \in \ZZ : k_1\alpha_1 + \cdots + k_n \alpha_n = 0) \implies (k_1=\cdots =k_n = 0),
\]
then every periodic trajectory $\Gamma$ in $Q$ is stable.

Although \cref{thm:main} produces families of irrational polygons, we cannot apply the foregoing result of Vorobets, Gal'perin, and Stepin to conclude the produced periodic trajectories in irrational polygons are stable.
The reason is that, in the proof, the edges $w_{a_1}$ and $w_{a_n}$ are rotated about points $z_1$ and $z_n$ simultaneously; we cannot thereby conclude that the interior angles of the polygons in the produced path are rationally independent.

\subsection*{Acknowledgements}
I would like to thank Sam Freedman for the many conversations that ultimately motivated me to write this paper.
This work was supported by the National Science Foundation Graduate Research Fellowship Program under Grant No. 2140001.

\bibliographystyle{alpha}
\bibliography{references}

\end{document}